\documentclass[12pt]{ijnam}
\copyrightinfo{2012}{} 

\usepackage{amsmath, eucal, pstcol, graphicx, picinpar, mathrsfs}
\usepackage{epsfig}
\usepackage{amsmath,amssymb}
\usepackage{graphicx}
\usepackage{setspace}

\usepackage{amsmath,amsthm}
\usepackage{latexsym}
\newtheorem{theorem}{Theorem}[section]

\newtheorem{lemma}[theorem]{Lemma}
\newtheorem{remark}[theorem]{Remark}
\usepackage{hhline}
\begin{document}

\title[Parabolic reaction-diffusion systems]{ High Order Parameter-Robust Numerical Method for  a System of $(M\geq 2)$  Coupled Singularly Perturbed Parabolic Reaction-Diffusion Problems}
\author[ M. Kumar ]{Mukesh Kumar }
 \address{Differential Equations and Numerical
Analysis Group,
 Department of Mathematical Sciences,
Norwegian University of Sciences and Technology,
 NO-7491, Trondheim, Norway } \email{mukesh.kumar@math.ntnu.no }

\author[ S. C. S. Rao ]{S. Chandra Sekhara Rao }
 \address{
  Department of Mathematics,
  Indian Insitute of Technology Delhi,
  Hauz Khas, New Delhi-110016, India
} \email{scsr@maths.iitd.ernet.in }






\abstract{ We present a high order parameter-robust numerical method for a
 system of  $(M\geq 2)$ coupled  singularly perturbed parabolic
reaction-diffusion problems. A small  perturbation parameter
$\varepsilon $ is multiplied with the second order spatial
derivatives in all the equations. The parabolic boundary layer
appears in the solution of the problem when the perturbation
parameter $\varepsilon$ tends to zero. To obtain a high order approximation to the solution of this problem, we propose a numerical method  that employs the Crank-Nicolson method on an uniform mesh in  time
direction, together with a hybrid finite difference scheme on a generalized Shishkin mesh
in spatial direction. We prove that the resulting method
  is parameter-robust or $\varepsilon$-uniform  of second
order in time and almost fourth order in spatial variable, if the
discretization parameters satisfy a non-restrictive relation.
Numerical experiments are presented to validate
 the theoretical results and also indicate that the relation between the discretization parameters is
 not necessary in practice.}


\maketitle
\section{Introduction}
We consider the following system of $(M\ge2)$ coupled  singularly
perturbed parabolic reaction-diffusion problems
\begin{equation}\label{a1}
\textbf{\emph{L}}_{\varepsilon}\textbf{\emph{u}}:=\frac{\partial
\textbf{\emph{u}}}{\partial
t}+\textbf{\emph{L}}_{x,\varepsilon}\textbf{\emph{u}}=\textbf{\emph{f}},~~~~~~(x,t)\in
D:=\Omega\times(0,T]=(0,1)\times(0,T],
\end{equation}
\begin{equation}\label{a2}
\textbf{\emph{u}}(0,t)=\mathbf{0},~~~~~\textbf{\emph{u}}(1,t)=\mathbf{0},~~~~~~\forall
t\in [0,T],~~~~~~~~~~~~~~~~~~~~~~~~~~~~~~~~~~~~~~~~~~
\end{equation}
\begin{equation}\label{a3}
\textbf{\emph{u}}(x,0)=\mathbf{0},~~~~~~~~\forall x\in
\overline{\Omega}.~~~~~~~~~~~~~~~~~~~~~~~~~~~~~~~~~~~~~~~~~~~~~~~~~~~~~~~~~~~~~~~~~~~
\end{equation}
 \noindent The spatial differential operator $\textbf{\emph{L}}_{x,\varepsilon}$ is defined by
$$\hspace*{-0.0cm} \textbf{\emph{L}}_{x,\varepsilon}=\left(%
\begin{array}{cccc}
 -\varepsilon\frac{\partial^2 }{\partial x^2} &    & 0 \\
     &  \ddots&  \\
   0 &  & -\varepsilon\frac{\partial^2 }{\partial x^2}\\
\end{array}%
\right) + \textbf{\emph{A}},~\mbox{with}~\textbf{\emph{A}}=\left(%
\begin{array}{cccc}
  a_{11}(x) &  ...&a_{1M}(x) \\
 \vdots   &\ddots& \vdots \\
 a_{M1}(x)& ...& a_{MM}(x) \\
\end{array}%
\right),$$ \noindent where $\varepsilon$ is a  small parameter that 
satisfies $0<\varepsilon \ll1$. Denote the boundaries of the domain
$D$ by $\Gamma:=\Gamma_0\bigcup
\Gamma_1$, with  $\Gamma_0=\{(x,0)|x\in \Omega\}$ and $\Gamma_1=\{(x,t)|x=0,1,~t\in[0,T]\}$. We assume that the coupling matrix
$\textbf{\emph{A}}=(a_{ij}(x))_{M\times M}$ satisfies the following
positivity conditions at each  $x\in\overline{\Omega}$
\begin{equation}\label{eq:assump1}
a_{ij}\leqslant 0,~~~i\neq j,~~~~
\end{equation}
\begin{equation}\label{eq:assump2}
a_{ii}>
0,~~~\displaystyle\sum_{j=1}^{M}a_{ij}\ge\beta^*>0,~~~i=1,\dots,M.
\end{equation}
If  $(5)$ is not satisfied directly, we consider the transformation
$\widetilde{\textbf{\emph{u}}}(x,t)=\textbf{\emph{u}}(x,t)\exp(-\beta_0t)$
with $\beta_0>0$ (sufficiently large) in order to transform the
diagonal entries such that $(5)$ holds. Also, we  assume that
sufficient regularity and compatibility conditions  hold among the
data of the problem $(1)$-$(3)$  such that the exact solution
$\textbf{\emph{u}}\in C^{6,3}(\overline{D})^M$. In the analysis we assume the following compatibility conditions
(see \cite{lady})
\[\frac{\partial^{s+q}\textbf{\emph{f}}}{\partial x^s \partial t^q}(0,0)=\frac{\partial^{s+q}\textbf{\emph{f}}}{\partial x^s \partial t^q}(1,0)=\mathbf{0},~~~\mbox{for}~~0\leqslant s+2q\leqslant 4.\]

The numerical analysis of singular perturbation problems has always
suffered  from serious difficulties due to the boundary layer
behavior of the solution when the perturbation parameter becomes
small. Recent years have witnessed substantial progress in the
development of layer adapted meshes to design a special class of
numerical methods, so called \emph{parameter-robust numerical
methods}, that converge uniformly with respect to the perturbation
parameter (see \cite{roos96}). Parameter-robust numerical methods based on fitted meshes, particularly the Shishkin meshes gained popularity because of their simplicity and applicability to more
complicated problems in higher dimensions, see \cite{kopteva} for more details. Several numerical studies for coupled
system of  singularly perturbed reaction-diffusion problems are
considered in \cite{linss04},\cite{linss08},\cite{mad},\cite{mat},\cite{scsr}
and the references therein.\\

 To solve the system of two coupled
singularly perturbed parabolic reaction-diffusion problems with the
distinct small perturbation parameters in each equations, Gracia and
Lisbona \cite{gra07 } proposed a uniformly convergent numerical
method by using the classical backward Euler scheme in time and the central
difference scheme in spatial direction, and  proved that the error bound  is
$O(\Delta t + N^{-2+q}\ln^2 N)$ with the assumption $N^{-q}\leqslant
C\Delta t$, ~$0<q<1$. High order numerical methods have always been an interest for the numerical community as they provide good numerical approximations with low computational cost.
 Recently, Clavero et al. \cite{clav06} gave an attempt to design a
high order uniformly convergent numerical method for solving the
system of two coupled singulary perturbed parabolic
reaction-diffusion problems with the distinct small perturbation
parameters  in each equations. To increase the order of uniform
convergence, the authors in \cite{clav06} considered the
Crank-Nicolson method on an uniform mesh in time direction and central
difference scheme on a standard Shishkin mesh in spatial direction, and
proved that the error bound  is $O((\Delta t)^2 + N^{-2+q}\ln^2 N)$
with the assumption $N^{-q}\leqslant C\Delta t$, ~$0<q<1$. To our
knowledge this is the only high order parameter-robust numerical
method is available in the literature for solving  parabolic reaction-diffusion system $(1)$-$(3)$. In the present paper, our objective is to integrate the available techniques for high order approximations (eg. \cite{clav06} and \cite{kumar}), to design a high order parameter-robust numerical method for solving parabolic reaction-diffusion system $(1)$-$(3)$. For a high order approximation, we consider  the Crank-Nicolson method  on an uniform
mesh in time, together with a hybrid scheme which is a suitable
combination of the fourth order compact difference scheme and the
standard central difference scheme on a generalized Shishkin mesh in
spatial direction. It can be seen that the combination of
Crank-Nicolson method in time direction with hybrid scheme in
spatial direction does not satisfy the discrete maximum principle
except if the restrictive condition $\Delta t \leqslant
C(L/N)^2$ is imposed. In this article,  we follow the approach of Clavero et al. \cite{clav05} to overcome this difficulty. First, 
some auxiliary problems are considered which permits to prove appropriate bounds
for  local error of the Crank-Nicolson method. Then the uniform
convergence analysis of the scheme used to discretize these
auxiliary problems is discussed. Finally, using the recursive
arguments and the uniform stability of the totally discrete scheme, we
claim that the present method is uniformly convergent of second
order  in time and almost fourth order  in spatial variable. It
should be noted here that in the theoretical proof  we assume the
totally discrete scheme operator satisfy the uniform stability as a
conjecture in Section 5. As so far it is an
open problem to prove the uniform stability of totally discrete
scheme theoretically (see also  \cite{clav06}). While in the support
we presented the numerical tables (Tables 2,4 and 6) that shows the spectral radius of the
totally discrete operator is strictly less than one, independent of $\varepsilon$ and discretization  parameters, in Section 6. \\

 This paper is arranged as follows. In Section 2, a  priori bounds on the solution of $(1)$-$(3)$ and its
 derivative are constructed.
The time semidiscretization  using the Crank-Nicolson method  and its local consistency error 
is given in Section 3. In this section we also discuss the
asymptotic behavior of the solution of semidiscretized problems and their
spatial derivatives. In Section 4, the generalized Shishkin mesh is
given and the spatial semidiscretization  with a hybrid scheme
which is a suitable combination of the fourth order compact
difference scheme and the central difference scheme is described  on generalized Shishkin mesh for the set of stationary singularly
perturbed problems studied in Section 3. It is also proved that the
spatial semidiscretization  is  almost fourth order uniformly
convergent on generalized Shishkin mesh. In Section 5, semidiscretization steps are combined to give the total discretization  and its uniform convergence is proved. The numerical experiments are conducted
to demonstrate the efficiency of the proposed method in Section 6. Finally,
 conclusions are included in Section 7.
\\

\noindent\textbf{Notations}:~~In the remaining parts of the paper,
$C$ and $\emph{\textbf{C}}=C(1,\ldots,1)^T$ are the generic positive
constant independent of $\varepsilon$ and discretization parameters.
Define $\textbf{\emph{v}}\leqslant\textbf{\emph{w}}$~ if ~$v_i\leqslant
w_i,~1\leqslant i\leqslant M$ and $|\textbf{\emph{v}}|=(|v_1|,\dots,|v_M|)^T$.
We consider the maximum norm and  it is denoted by $||.||_H$, where
$H$ is a closed and bounded set. For a real valued function $v\in
C(H)$ and for a vector valued function
$\textbf{\emph{v}}=(v_1,\dots,v_M)^T\in C(H)^M$, we define
$$||v||_H=\max_{x\in H}|v(x)|~~~~~~~~\mbox{and}~~~~~~~||\textbf{\emph{v}}||_H=\max\{||v_1||_H,\dots,||v_M||_H\}.$$
If $H=\overline{\Omega}$, we drop $H$ from the notation. The
analogous discrete maximum norm on the mesh $\overline{\Omega}^S_N$
is denoted by $||.||_{\overline{\Omega}^S_N}$. For any function
$g\in C(\overline{\Omega})$, $g_i$ is used for $g(x_i)$; if
$\textbf{g}\in C(\overline{\Omega})^M$ then
$\textbf{g}_i=\textbf{g}(x_i)=(g_{1,i},\dots,g_{M,i})^T.$ ~For
simplicity, we use $L_{N_0}$ for $L(N_0)$.  If $N_0=N$, we drop $N$
as subscript from the notation $L_N$.

\section{Properties of the exact solution}

Following the technique of Theorem 1 in \cite{gra07 }, we can show
that the operator $\textbf{\emph{L}}_{\varepsilon}$ in (1) satisfies the
following maximum principle.

\begin{lemma}\label{tslem:cmp}
Let $\textbf{y}\in (C^{2,1}(D)\cap C^{0,0}(\overline{D}))^M.$ Let
$\textbf{y}(x,0)\geq \mathbf{0}$ on $\overline{\Omega}$ and
$\textbf{\emph{y}}(0,t)\geq \mathbf{0},$  $\textbf{y}(1,t)\geq \mathbf{0}$
on $[0,T].$ Then $\textbf{L}_{\varepsilon}\textbf{y}\geq \mathbf{0}$
in $D$ implies $\textbf{y}\geqslant \mathbf{0}$ on $\overline{D}.$
\end{lemma}

An immediate consequence of Lemma 2.1 is the following stability result.
\begin{lemma}\label{tslem:pus}
Let  $\textbf{u}$ be the solution of  (1)-(3). Then
\[||\textbf{u}||_{\overline{D}}\leqslant \frac{1}{\beta^*}||\textbf{f}\,||_{\overline{D}}.\]
\end{lemma}
To obtain the bounds on the solution $\textbf{\emph{u}}$ of (1)-(3),
the variable $x$ is transformed to the stretched variable
$\widetilde{x}$ ~defined by $\widetilde{x}=x/\sqrt{\varepsilon}$,
this results that $(1)$-(3) transformed as

\begin{equation}\label{2.3a}
\widetilde{\textbf{\emph{L}}}_{\varepsilon}\widetilde{\textbf{\emph{u}}}:=\frac{\partial
\widetilde{\textbf{\emph{u}}}}{\partial
t}+\widetilde{\textbf{\emph{L}}}_{\widetilde{x},\varepsilon}\widetilde{\textbf{\emph{u}}}=\widetilde{\textbf{\emph{f}}},~~~~~~(\widetilde{x},t) \in~\widetilde{D}_{\varepsilon},
\end{equation}

\begin{equation}\label{2.3b}
\widetilde{\textbf{\emph{u}}}(\widetilde{x},t)=0,~~(\widetilde{x},t)\in\widetilde{\Gamma}_{\varepsilon},~~~~~~~~~~~~~~~~~~~~~~
\end{equation}
 where
$$ \widetilde{\textbf{\emph{L}}}_{\widetilde{x},\varepsilon}=\left(%
\begin{array}{cccc}
 -\frac{\partial^2 }{\partial \widetilde{x}^2} &    & 0 \\
     & \ddots &  \\
   0 &  & -\frac{\partial^2 }{\partial \widetilde{x}^2}\\
\end{array}%
\right) + \widetilde{\textbf{\emph{A}}}, ~\mbox{with}~\widetilde{\textbf{\emph{A}}}=\left(%
\begin{array}{cccc}
  \widetilde{a}_{11}(\widetilde{x}) & ...&\widetilde{a}_{1M}(\widetilde{x}) \\
 \vdots  & \ddots & \vdots\\
 \widetilde{a}_{M1}(\widetilde{x}) & ...& \widetilde{a}_{MM}(\widetilde{x}) \\
\end{array}%
\right),$$
$\widetilde{D}_{\varepsilon}=\widetilde{\Omega}_{\varepsilon}\times(0,T]=(0,1/\sqrt{\varepsilon})\times
(0,T]$ and $\widetilde{\Gamma}_{\varepsilon}$ is its boundary
analogous to $\Gamma$. Here the differential equation   $(6)$ is
independent of $\varepsilon$. Using the standard local estimate for
the solution of system of time dependent partial differential
equations (see \cite{lady}),  we obtain the bounds on the solution
of $(6)$-$(7)$ and its derivative. On returning  in term to the
original variable $x$  and using $||u||_{\overline{D}}\leqslant C$,  obtained from 
Lemma 2.2 with  $\varepsilon$-uniform boundedness of
$\textbf{\emph{f}}$, yields the following result.

\begin{lemma}Let $\textbf{u}$ be the solution of (1)-(3).
Then it satisfies
$$\left\|\frac{\partial^{i+j}\textbf{u}}{\partial x^i \partial t^j}\right\|_{\overline{D}}\leqslant C \varepsilon^{-i/2}, ~~~\mbox{for}~~~0\leqslant i+2j\leqslant 6.\\$$
\end{lemma}

In the following result, we derive sharper bounds on the derivatives
of $\textbf{\emph{u}}$ to show that the large values seen in Lemma
2.3 do in fact decay rapidly as one moves away from the boundary
$\Gamma$.
\begin{lemma}\label{tslem:bnd2}
Let $\textbf{u}$ be the solution of (1)-(3).  Let
$\beta\in(0,\beta^*)$ be arbitrary but has a fixed value. Then there exists a
constant $\textbf{C}$, independent of $\varepsilon$, such that
\begin{equation}\label{tseq:bnd3}
\left|  \frac{\partial^m \textbf{u}(x,t)}{\partial x^m}\right|\leqslant
\textbf{C}\left(1+\varepsilon^{-m/2}(\exp(-x\sqrt{\beta/\varepsilon})+\exp(-(1-x)\sqrt{\beta/\varepsilon}))\right)
\end{equation}
for $(x,t)\in\overline{D}$ and $m=0,\dots,6.$
\end{lemma}
\begin{proof} Fix $\beta\in(0,\beta^*)$ and  set
$P_m(x)=1+\varepsilon^{-m/2}(\exp(-x\sqrt{\beta/\varepsilon})+\exp(-(1-x)\sqrt{\beta/\varepsilon})).$
The proof is by mathematical induction.
 The bound (\ref{tseq:bnd3}) for $m=0$ follows from Lemma 2.2. Assume that (\ref{tseq:bnd3}) holds for
 $m=0,\dots,\nu-1,$ $1\leqslant \nu\leqslant 6.$ We now prove (\ref{tseq:bnd3})
 for $m=\nu.$ Letting $$\textbf{\emph{y}}=\frac{\partial^{\nu}\textbf{\emph{u}}}{\partial x^{\nu}},$$ note that
\begin{equation*}
\left\{%
\begin{array}{ll}
    \frac{\partial \textbf{\emph{y}}}{\partial t}-\textbf{\emph{E}} \frac{\partial^2\textbf{\emph{y}}}{\partial x^2}+\textbf{\emph{A}}\textbf{\emph{y}}=
\frac{\partial^\nu \textbf{\emph{f}}}{\partial x^{\nu}}-\sum_{l=0}^{\nu-1}\left(%
\begin{array}{c}
  \nu \\
  l\\
\end{array}%
\right) \textbf{\emph{A}}^{(\nu-\ell)} \frac{\partial^{\ell}\textbf{\emph{u}}}{\partial x^{\ell}}:=\boldsymbol{\Psi}_\nu & \mbox{in}~D, \\
    \textbf{\emph{y}}(x,0)=\mathbf{0} & \mbox{in}~ \overline{\Omega},\\
    ||\textbf{\emph{y}}(0,t)||\leqslant C\varepsilon^{-\nu/2}, ||\textbf{\emph{y}}(1,t)||\leqslant C\varepsilon^{-\nu/2} & \mbox{in}~(0,T], \\
\end{array}%
\right.
\end{equation*}
where boundary conditions follow from Lemma 2.3. From
the inductive hypothesis, it is clear that
$|\boldsymbol{\Psi}_\nu(x,t)|\leqslant \textbf{C}P_{\nu-1}(x).$ Applying the
maximum principle with the barrier function
$\textbf{\emph{C}}P_\nu(x),$ we obtain the required result, i.e.,  for $(x,t)\in\overline{D}$
\[\left| \frac{\partial^{\nu} \textbf{\emph{u}}(x,t)}{\partial x^{\nu}}\right|\leqslant
\textbf{\emph{C}}\left(1+\varepsilon^{-\nu/2}(\exp(-x\sqrt{\beta/\varepsilon})+\exp(-(1-x)\sqrt{\beta/\varepsilon}))\right).\]
\noindent This proves the lemma.
\end{proof}

Now a special decomposition of the exact solution $\textbf{\emph{u}}$
into a regular part $\textbf{\emph{v}}$ and a layer part
$\textbf{\emph{w}}$ can be obtain as follow. Set
$x^*=4\sqrt{\varepsilon / \beta } \ln(1/\sqrt{\varepsilon})$. Define for each $\jmath \in\{1,\dots,n\}$ and $(x,t)\in\overline{D}$
\begin{equation}
v_\jmath(x,t)=\left\{%
\begin{array}{ll}
   \displaystyle \sum_{\nu=0}^{4}\frac{(x-x^*)^\nu}{\nu!}\partial_x^{\nu}u_\jmath(x^*,t) &~ \mbox{for}~~ 0\leqslant x \leqslant x^*, t\in[0,T]; \\
    u_\jmath(x,t) &~ \mbox{for}~ x^*\leqslant x \leqslant 1-x^*, t\in[0,T];  \\
   \displaystyle \sum_{\nu=0}^{4}\frac{(x-x^*)^\nu}{\nu!}\partial_x^{\nu}u_\jmath(1-x^*,t)&~ \mbox{for}~ 1-x^*\leqslant x \leqslant 1, t\in[0,T], \\
\end{array}%
\right.
\end{equation}
and $w_\jmath(x,t)=u_\jmath(x,t)-v_\jmath(x,t)$. Then  Lemma
\ref{tslem:bnd2} and the choice of $x^*$ yields, for
$s=0,\dots,6$, (cf. Linss \cite{linss01})
\begin{equation*}\label{teq:rbond}
\left|\frac{\partial^sv_\jmath(x,t)}{\partial x^s} \right|\leqslant
C(1+\varepsilon^{2-s/2})~~~~~~~~~~~~ \tag{10a}
\end{equation*}
\begin{equation*}\label{teq:lbond}
\left|\frac{\partial^sw_\jmath(x,t)}{\partial x^s} \right|\leqslant
C\varepsilon^{-s/2}\left(\exp(-x\sqrt{\beta/\varepsilon})+\exp(-(1-x)\sqrt{\beta/\varepsilon})\right).~~~\tag{10b}
\end{equation*}
\setcounter{equation}{10} It should be noted here that this decomposition does not, in general,
satisfy $\textbf{\emph{L}}_{\varepsilon}\textbf{\emph{v}} =
\textbf{\emph{f}}$ and $\textbf{\emph{L}}_{\varepsilon}\textbf{\emph{w}}
=\mathbf{0}$.

\section{The time semidiscretization}

 We introduce the time semidiscretization of $(1)$-(3)  by using the classical Crank-Nicolson method,
 with  constant time step $\Delta t$ on  uniform mesh $\varpi=\{n\Delta t, 0\leqslant n \leqslant T/ \Delta t \}$. The time semidiscretization
 is given by 
\begin{equation}\label{3.1}
\left\{%
\begin{array}{ll}
\emph{\textbf{u}}^0=\emph{\textbf{u}}(x,0)=\textbf{0},\\
(I+\frac{\Delta t}{2}\textbf{\emph{ L}}_{x,\varepsilon})\textbf{\emph{u}}^{n+1}=(I-\frac{\Delta t}{2}  \textbf{\emph{L}}_{x,\varepsilon})\textbf{\emph{u}}^n+\frac{\Delta t}{2} (\textbf{\emph{f}}^n+\textbf{\emph{f}}^{n+1}),\\
\textbf{\emph{u}}^{n+1}(0)=\textbf{\emph{u}}^{n+1}(1)=\textbf{0},~\mbox{for}~~ n=0,1,\ldots, T/ \Delta t-1,
\end{array} \right.\
\end{equation}
where $\textbf{\emph{u}}^n$  is the   approximation of the exact
solution $\textbf{\emph{u}}$ of (1)-(3) at the time level
$t_n=n\Delta t, ~n=0,1,2,\ldots,T/\Delta t$ and
$\textbf{\emph{f}}^n=\textbf{\emph{f}}(x,t_n)$. \\

To study the consistency of (\ref{3.1}), we define the following
auxiliary  problem
\begin{equation}\label{3.3}
\left\{%
\begin{array}{ll}
\widehat{\textbf{\emph{L}}}_{x,\varepsilon} \widehat{\textbf{\emph{u}}}^{n+1} :=(I+\frac{\Delta t}{2} \textbf{\emph{L}}_{x,\varepsilon})\widehat{\textbf{\emph{u}}}^{n+1}=(I-\frac{\Delta t}{2}  \textbf{\emph{L}}_{x,\varepsilon})\textbf{\emph{u}}(x,t_n)+\frac{\Delta t}{2} (\textbf{\emph{f}}^n+\textbf{\emph{f}}^{n+1}),\\
~\widehat{\textbf{\emph{u}}}^{n+1}(0)=\widehat{\textbf{\emph{u}}}^{n+1}(1)=\textbf{0},
\end{array} \right.\
\end{equation}
 where
$\widehat{\textbf{\emph{u}}}^{n+1}$ is the approximation to
$\textbf{\emph{u}}(x,t_{n+1})$. Let
$\textbf{\emph{e}}_{n+1}(x)=\textbf{\emph{u}}(x,t_{n+1})-\widehat{\textbf{\emph{u}}}^{n+1}(x)$ be the local
truncation error of (\ref{3.1}) and it satisfies the following
lemma.

\begin{lemma}
If ~~~$$\left|\frac{\partial^i \textbf{u}(x,t)}{\partial t^i}\right|\leqslant \textbf{C},
~~~~(x,t)\in \overline{D}, ~~0\leqslant i \leqslant 3,$$ \noindent then the
local error associated to the scheme (\ref{3.1}) satisfies
$$|\textbf{e}_{n+1}(x)| \leqslant
\textbf{C}(\Delta t)^3,~~~x\in \overline{\Omega}.$$
\end{lemma}
\begin{proof}
The results follows from the arguments given in \cite{clav05}.
\end{proof}

Now we prove that  the asymptotic behavior of the solution of the
semidiscretize problem  (\ref{3.3}) and its spatial derivative have
essentially the same asymptotic behavior that the solution of a
stationary  system of  coupled singularly perturbed
reaction-diffusion problems. Using the approach Clavero et al.\cite{clav}, such feature is given by the following
lemma.\\

\begin{lemma}
Let $\widehat{\textbf{u}}^{n+1}$ be the solution of (\ref{3.3}).
Then it satisfies
\begin{equation}\label{3.4}
\left|\frac{d^{k}\widehat{\textbf{u}}^{n+1}}{d x^k }\right|\leqslant
\textbf{C}\left(1+\varepsilon^{-k/2}
(\exp(-x\sqrt{\beta/\varepsilon})+\exp(-(1-x)\sqrt{\beta/\varepsilon})\right),
\end{equation}
where $0\leqslant k\leqslant6$ and $\textbf{C}$ is a constant independent of
$\varepsilon$ and $\Delta t.$
\end{lemma}

\begin{proof}
Let us first start by studying the behaviour of
$\widehat{\textbf{\emph{u}}}^{n+1}$, that means the result (\ref{3.4})  for $k=0.$
As the data $\textbf{\emph{f}}$ is $\varepsilon$-uniformly bounded,
$|\textbf{\emph{u}}(x,t_n)|\leqslant \textbf{\emph{C}}$ and
$|\textbf{\emph{L}}_{x,\varepsilon}\textbf{\emph{u}}(x,t_n)|\leqslant \textbf{\emph{C}}$;
similar to \cite{mad},  the operator $(I+\frac{\Delta
t}{2}\textbf{\emph{L}}_{x,\varepsilon})$ satisfies a maximum
principle and using this it follows that
$$|\widehat{\textbf{\emph{u}}}^{n+1}|\leqslant \textbf{\emph{C}}.$$

To prove the result (\ref{3.4}) for the derivatives of
$\widehat{\textbf{\emph{u}}}^{n+1}$, we introduce the following
auxiliary function
$$\mathbf{\phi}^{n+1}=\frac{2}{\Delta t}(\widehat{\textbf{\emph{u}}}^{n+1}(x)-\textbf{\emph{u}}(x,t_n)),$$
which is the solution of the following boundary value problem
\begin{equation}\label{3.5}
\left\{%
\begin{array}{ll}
(I+\frac{\Delta
t}{2}\textbf{\emph{L}}_{x,\varepsilon})\phi^{n+1}=-2\textbf{\emph{L}}_{x,\varepsilon}\textbf{\emph{u}}(x,t_n)+\textbf{\emph{f}}^n+\textbf{\emph{f}}^{n+1},\\
\\
\phi^{n+1}(0)=\textbf{0},~~~~~~~\phi^{n+1}(1)=\textbf{0}.
\end{array} \right.\
\end{equation}
Using
$|\textbf{\emph{L}}_{x,\varepsilon}\textbf{\emph{u}}(x,t_n)|=|\textbf{\emph{f}}(x,t_n)-\frac{\partial
\textbf{\emph{u}}}{\partial t}(x,t_n)|\leqslant \textbf{\emph{C}}$ and
$|\phi^{n+1}(0)|\leqslant \textbf{\emph{C}}, ~~|\phi^{n+1}(1)|\leqslant \textbf{\emph{C}} $ with the
maximum principle for $(I+\frac{\Delta
t}{2}\textbf{\emph{L}}_{x,\varepsilon})$  we get
$$|\phi^{n+1}|\leqslant \textbf{\emph{C}}.$$
Next we write the problem (\ref{3.3}) as
\begin{equation}\label{3.6}
\left\{%
\begin{array}{ll}
\textbf{\emph{L}}_{x,\varepsilon}\widehat{\textbf{\emph{u}}}^{n+1}=- \phi_1^{n+1}-\textbf{\emph{L}}_{x,\varepsilon}\textbf{\emph{u}}(x,t_n)+\textbf{\emph{f}}^n+\textbf{\emph{f}}^{n+1},\\
\\
\widehat{\textbf{\emph{u}}}^{n+1}(0)=\textbf{0},~~~~~~~\widehat{\textbf{\emph{u}}}^{n+1}(1)=\textbf{0}.
\end{array} \right.\
\end{equation}
From $|\phi_1^{n+1}|\leqslant \textbf{\emph{C}}$, it can be seen  that the right side of
(\ref{3.6}) is $\varepsilon$-uniformly bounded. Using this with
$|\widehat{\textbf{\emph{u}}}^{n+1}|\leqslant \textbf{\emph{C}}$ we get
\begin{equation}\label{3.7}
\left\|\frac{d^2\widehat{\textbf{\emph{u}}}^{n+1}}{dx^2}\right\|_{\overline{\Omega}}\leqslant
C  \varepsilon^{-1}, ~~~~~~~
\end{equation}
From (\ref{3.7}) and using the mean value theorem argument as used
in \cite{mad}, we  obtain
\begin{equation}\label{3.8}
\left\|\frac{d\widehat{\textbf{\emph{u}}}^{n+1}}{dx}\right\|_{\overline{\Omega}}\leqslant
C \varepsilon^{-1/2}. ~~~~~~~
\end{equation}
On differentiating (\ref{3.3}) with respect to $x$, we define
$\zeta_i^{n+1}=\frac{d^i\widehat{\textbf{\emph{u}}}^{n+1}}{dx^i},
~i=1,2,$ are  the solutions of boundary value problems
\begin{equation}\label{3.9}
\left\{%
\begin{array}{ll}
(I+\frac{\Delta t}{2}\textbf{\emph{L}}_{x,\varepsilon})\zeta_i^{n+1}=\textbf{\emph{g}}_i(x),\\
\\
\zeta_i^{n+1}(0)=\textbf{\emph{s}}_i^0,~~~~~~~\zeta_i^{n+1}(1)=\textbf{\emph{s}}_i^1,\\
\end{array} \right.\
\end{equation}
where $|\textbf{\emph{s}}_i^j|\leqslant \textbf{\emph{C}} \varepsilon^{-i/2},~~~i=1,2,
~j=0,1$ and using Lemma 2.4  $$|\textbf{\emph{g}}_i(x)| \leqslant \textbf{\emph{C}} (1+\varepsilon^{-i/2}
(\exp(-x\sqrt{\beta/\varepsilon})+\exp(-(1-x)\sqrt{\beta/\varepsilon}))),~i=1,2.$$  Now taking the barrier function as
$$\widetilde{\xi}(x)=\textbf{\emph{C}}_1(1+x)+\textbf{\emph{C}}_2
\varepsilon^{-i/2}(\exp(-x\sqrt{\beta/\varepsilon})+\exp(-(1-x)\sqrt{\beta/\varepsilon}))$$
and for sufficiently large value of $\textbf{\emph{C}}_1$ and $\textbf{\emph{C}}_2$ using the
maximum principle for $(I+\frac{\Delta
t}{2}\textbf{\emph{L}}_{x,\varepsilon})$, we deduce that (\ref{3.4})
is
true for $k=1,2.$\\

Now to prove the bound (\ref{3.4}) for higher value of $k$, we 
follow similar arguments given in \cite{clav}. This proves the lemma.
\end{proof}

Next we define the Shishkin-type decomposition for the solution of
semidiscretize problem (\ref{3.3}). This type of decomposition has
been discussed earlier in Linss \cite{linss01}, for scalar singularly
perturbed boundary value problem. To define this, let
$x^*=(4\sqrt{\varepsilon/ \beta}) \ln(1/\sqrt{\varepsilon})$. Similar to (9), for each
 $x\in \overline{\Omega}$ and $k=1,\ldots,M$, we set $\widehat{\emph{v}}^{n+1}_k=\widehat{\emph{u}}^{n+1}_k$
for $x\in [x^*,1-x^*]$ and $\widehat{\textbf{\emph{v}}}^{n+1}$
extends to a smooth function defined on $\overline{\Omega}$ and  define
$\widehat{\emph{w}}^{n+1}_k=\widehat{\emph{u}}^{n+1}_k-\widehat{\emph{v}}^{n+1}_k$
for $x\in \overline{\Omega}.$  Then the results of Lemma 3.2 and the
choice of $x^*$ implies the following
 decomposition of $\widehat{\textbf{\emph{u}}}^{n+1}$  (cf. Linss \cite{linss01})
\begin{lemma}
Let $\widehat{\textbf{u}}^{n+1}$ be the solution of (\ref{3.3}).
Then it can be represented as
$\widehat{\textbf{u}}^{n+1}=\widehat{\textbf{v}}^{n+1}+\widehat{\textbf{w}}^{n+1}$,
where the regular part $\widehat{\textbf{v}}^{n+1}$ satisfies
\begin{equation}\label{3.20}
\left|\frac{d^{m}\widehat{v}^{n+1}_k}{d x^m }\right|\leqslant
C(1+\varepsilon^{2-m/2}),
\end{equation}
and the layer part $\widehat{\textbf{w}}^{n+1}$ satisfies
\begin{equation}\label{3.21}
\left|\frac{d^{m}\widehat{w}^{n+1}_k}{d x^m}\right|\leqslant
C\varepsilon^{-m/2} \left(\exp(-x\sqrt{\beta/\varepsilon})+\exp(-(1-x)\sqrt{{\beta}/\varepsilon})\right),
\end{equation}
for ~$0\leqslant m\leqslant6$,~~ $k=1,\ldots, M, $~and
$C$ is a constant independent of $\varepsilon$ and $\Delta t$.
\end{lemma}

The above lemma shows that
 the solution
 $\widehat{\textbf{\emph{u}}}^{n+1}$  of (\ref{3.3}) is decomposed into a sum of  regular part $\widehat{\textbf{\emph{v}}}^{n+1}=(\widehat{v}^{n+1}_1,...,\widehat{v}^{n+1}_M)^T$ and  layer
part
$\widehat{\textbf{\emph{w}}}^{n+1}=(\widehat{w}^{n+1}_1,...,\widehat{w}^{n+1}_M)^T$,
that is, it can be written as
~$\widehat{\textbf{\emph{u}}}^{n+1}=\widehat{\textbf{\emph{v}}}^{n+1}+\widehat{\textbf{\emph{w}}}^{n+1}$.
This decomposition is said to be a Shishkin-type decomposition (not
a standard Shishkin decomposition) as it does not in general satisfy
$(I+\frac{\Delta
t}{2}\textbf{\emph{L}}_{x,\varepsilon})\widehat{\textbf{\emph{v}}}^{n+1}=(I-\frac{\Delta
t}{2}
\textbf{\emph{L}}_{x,\varepsilon})\textbf{\emph{v}}(x,t_n)+\frac{\Delta
t}{2} (\textbf{\emph{f}}^n+\textbf{\emph{f}}^{n+1})$ and
$(I+\frac{\Delta
t}{2}\textbf{\emph{L}}_{x,\varepsilon})\widehat{\textbf{\emph{w}}}^{n+1}=(I-\frac{\Delta
t}{2} \textbf{\emph{L}}_{x,\varepsilon})\textbf{\emph{w}}(x,t_n)$,
as these additional properties are not needed
in the error analysis of present method.

\section{The spatial semidiscretization}
In this section, first, we construct a
generalized Shishkin mesh $S(L)$ to discretized the domain
$\overline{\Omega}:=[0,1]$ by using a suitable mesh generating
function $\mathcal{K}$ as described in \cite{vu011}. Define the
transition parameter
$$\sigma=\mbox{min}\left\{~\frac{1}{4},~ \sigma_0\sqrt{\varepsilon}
L~\right\},$$ where $\sigma_0 (\geq 4/\sqrt{\beta})$ is a positive
constant and  $L=L(N)$ the value of $L$ with $N$ elements that
satisfy $\ln(\ln N)< L\leqslant \ln N$~~and
\begin{equation}\label{0}
e^{-L}\leqslant\frac{L}{N}.
\end{equation}

The mesh points of generalized-Shishkin discretized domain
$\overline{\Omega}_N^S$ are given by
$x_j=\mathcal{K}(j/N),\,j=0,1,\ldots,N/2,$ and by symmetry 
$x_{N-j}=1-x_j,\,j=0,1,...,N/2$, where $\mathcal{K}\in C^2[0,1/2]$ and defined as
\begin{equation}\label{4}
~~~~~~\mathcal{K}(t) = \left\{%
\begin{array}{ll}
  \displaystyle 4\sigma t,   ~~~~~~~~~~~~~~~~~~~~~~~~~~~~~~~~~~~~~~~~\mbox{for}~~ t\in [0,1/4]; \\
    [8pt]
   \displaystyle p(t-1/4)^3+4\sigma( t-1/4)+\sigma,  ~~~~~~\mbox{for}~~ t\in [1/4,1/2].
\end{array}%
\right.
\end{equation}
\noindent Here the  coefficient $p$ is determined
by  $\mathcal{K}(1/2)=1/2.$ \\

 Note that the mesh $\overline{\Omega}_N^S$ is uniform  in
$[0,\sigma]$ and $[1-\sigma, 1],$ and it changes smoothly in the
transition points $\{\sigma, 1-\sigma\}$. However, the mesh width
$h_j=x_{j+1}-x_j$, for $j=N/4,\ldots, 3N/4$,  satisfies (see
\cite{vu011})
\begin{equation}\label{4.a}
\begin{array}{lcl}
h_{j+1}\leqslant N^{-1} \displaystyle\max_{[(i-1)/N,(i+1)/N]}\mathcal{K}'(t)\leqslant C N^{-1}\\
h_{j+1}-h_j \leqslant N^{-2} \displaystyle\max_{[(i-1)/N,(i+1)/N]}\mathcal{K}^{''}(t)\leqslant C N^{-2}.
\end{array}
\end{equation}
We shall let
$h_{\max}=\displaystyle\max_{ \forall j}~ h_j, j=1,2,\ldots,N$, and by symmetry it is easy to verify that $h_{\max}=h_{N/2}=h_{N/2+1}$.

\subsection{The hybrid scheme}
We introduce a hybrid scheme  to discretize the set of stationary
coupled system of  singulary perturbed reaction-diffusion problem
(\ref{3.3}) on the generalized Shishkin mesh
$\overline{\Omega}_N^S$. The hybrid scheme is a combination of the
fourth order compact difference scheme (where the coefficients
$q_i^{k}$'s and $r_i^{k}$'s of the scheme are determined so that the
scheme is exact for the polynomials up to degree four ~and satisfy
the normalization conditions $q_i^{k,-}+q_i^{k,c}+ q_i^{k,+}=1, ~
i=1,2, \ldots, N-1,~ k=1,2, \ldots, M$ ) and the central difference
scheme, and is given
by
\begin{equation}\label{4.5a}
\hspace*{1.5cm}[\widehat{\textbf{\emph{L}}}^N_{x,\varepsilon}\widehat{\textbf{\emph{U}}}^{n+1}]_i=[\emph{\textbf{Q}}\widehat{\textbf{\emph{f}}}^n]_i,~~\mbox{for}~~i=1,2,\ldots,N-1,
\end{equation}
\begin{equation}\label{4.5b}
\hspace*{-1.85cm}\widehat{\textbf{\emph{U}}}^{n+1}_0=\textbf{0},~~~
\widehat{\textbf{\emph{U}}}^{n+1}_N=\textbf{0},~~~~~~~~~~~~~~~~~~
\end{equation}
 where
\begin{equation}\label{discoper}
[\widehat{\textbf{\emph{L}}}^N_{x,\varepsilon}\widehat{\textbf{\emph{U}}}^{n+1}]_i:=\left(%
\begin{array}{c}
  R(\widehat{U}^{n+1}_1)~+~\frac{\Delta t}{2}Q(a_{12}\widehat{U}^{n+1}_2)+...+\frac{\Delta t}{2}Q(a_{1M}\widehat{U}^{n+1}_M)\\

R(\widehat{U}^{n+1}_2)~+~\frac{\Delta
t}{2}Q(a_{21}\widehat{U}^{n+1}_1)+...+\frac{\Delta
t}{2}Q(a_{2M}\widehat{U}^{n+1}_M)\\
\vdots\\
R(\widehat{U}^{n+1}_M)~+~\frac{\Delta
t}{2}Q(a_{M1}\widehat{U}^{n+1}_1)+...+\frac{\Delta
t}{2}Q(a_{MM-1}\widehat{U}^{n+1}_{M-1})
\end{array}%
\right)_i,
\end{equation}

\begin{equation}
[\emph{\textbf{Q}}\widehat{\textbf{\emph{f}}}^n]_i:=\left(%
\begin{array}{c}
  Q({\widehat{f}}^n_1) \\

  Q({\widehat{f}}^n_2) \\
  \vdots
  \\
  Q({\widehat{f}}^n_M)

\end{array}%
\right)_i,
\end{equation}

\begin{equation}
\left\{%
\begin{array}{lcl}
R(V_{k,i}) &=& r_i^{k,-}V_{k,i-1}+r_i^{k,c}V_{k,i}+r_i^{k,+}V_{k,i+1},\\
Q(V_{k,i}) &=& q_i^{k,-}V_{k,i-1}+q_i^{k,c}V_{k,i}+q_i^{k,+}V_{k,i+1}.
\end{array}\right.
\end{equation}

\noindent with\\

$\widehat{\textbf{\emph{f}}}^{n}(x_i)=\textbf{\emph{u}}(x_i,t_n)+\frac{\Delta
t}{2}(-[\textbf{\emph{L}}_{x,\varepsilon}\textbf{\emph{u}}](x_i,t_n)+\textbf{\emph{f}}(x_i,t_n)+\textbf{\emph{f}}(x_i,t_{n+1})).$\\

\noindent The  coefficients $r_i^{k,*},i=1,\dots N-1, k=1,2, \ldots,
M, ~ *=-,c,+$ are given by
\begin{equation}\label{4.6}
\left\{%
\begin{array}{lcl}
r_i^{k,-} &=& \frac{\Delta
t}{2}(\frac{-2\varepsilon}{h_i(h_i+h_{i+1})}+q_i^{k,-}(a_{k,k;i-1}+\frac{2}{\Delta
t}))\\
 r_i^{k,c} &=& \frac{\Delta
 t}{2}(q_i^{k,-}a_{k,k;i-1}+q_i^{k,c}a_{k,k;i}+q_i^{k,+}a_{k,k;i+1})-r_i^{k,-}-r_i^{k,+}+1\\
 r_i^{k,+} &=& \frac{\Delta
t}{2}(\frac{-2\varepsilon}{h_{i+1}(h_i+h_{i+1})}+q_i^{k,+}(a_{k,k;i+1}+\frac{2}{\Delta
t}))
\end{array} \right.
\end{equation}
\noindent The coefficients $q_i^{k,*},i=1,\dots,N-1,~k=1,2 \ldots,
M,
 ~~*=-,c,+$ are defined in two different ways.\\

$(i)$~For the mesh points
 located in $(0,\tau)\cup(1-\tau,1)$, the coefficients
 $q_i^{k,*},i=\{1,\dots,
N/4-1\}\cup\{3N/4+1,\dots,N-1\}, k=1,2, \ldots, M,
 ~*=-,c,+$, are given by
\begin{equation}\label{4.7}
q_i^{k,-}=\frac{1}{12},~~~~~~~~~
 q_i^{k,c}=\frac{5}{6},~~~~~~~~
 q_i^{k,+}=\frac{1}{12}.
\end{equation}

$(ii)$~For the mesh points located in $[\tau,1-\tau]$, depending on
the relation between $h_{\max}$ and $\varepsilon$,  the coefficients
$q_i^{k,*},i=1,\dots,N-1,~k=1,2, \ldots, M,
 ~*=-,c,+$ are defined in
two different cases. Define $\widehat{a}_{kk}=a_{kk}+2/\Delta t$ for $k=1,2, \ldots, M$.\\

 In the first case, when $\gamma h_{\max}^2||\widehat{a}_{kk}||_\infty\leqslant \varepsilon,
$ where $\gamma$ is a positive constant independent of $\varepsilon$
and $\Delta t$, the coefficients $q_i^{k,*},i=N/4,\dots,3N/4,
~k=1,2, \ldots, M,~*=-,c,+$,
are given by
\begin{equation}
q_j^{k,-}=\frac{2h_j-h_{j+1}}{6(h_j+h_{j+1})},~~~~~~~~q_j^{k,c}=\frac{5}{6},~~~~~~~~q_j^{k,+}=\frac{2h_{j+1}-h_j}{6(h_j+h_{j+1})}.
\end{equation}

 While in the second case, when
$\gamma h_{\max}^2||\widehat{a}_{kk}||_\infty> \varepsilon, $ where
$\gamma$ is a positive constant independent of $\varepsilon$ and
$\Delta t$, the coefficients $q_i^{k,*},i=N/4,\dots,3N/4,~k=1,2,
\ldots, M,~*=-,c,+$, are given by
\begin{equation}\label{4.8}
q_i^{k,-}=0,~~~~~~q_i^{k,c}=1,~~~~~~q_i^{k,+}=0.
\end{equation}

The above definition of coefficients $q_i^k$'s and $r_i^k$'s show
that the scheme (24)-(25) is defined by the fourth order compact
difference scheme within the boundary layer region
$(0,\tau)\cup(1-\tau,1)$.  While in the regular region
$[\tau,1-\tau]$, the  scheme $(24)$-$(25)$ is defined  by a modified high order non-equidistant  difference scheme when
$\gamma h_{\max}^2||\widehat{a}_{kk}||_\infty\leqslant \varepsilon$
 and  is defined by the central difference scheme when
$\gamma h_{\max}^2||\widehat{a}_{kk}||_\infty> \varepsilon $.  This
means that the scheme $(24)$-(25) considers the
 high-order approximation  only when the local mesh width is
small enough to give non-positive off-diagonal entries while at all
other mesh points the central difference scheme is used. This
combination leads to the following lemma.
\begin{lemma}
Let $\gamma=1/6$ and $N_0$ be the smallest positive integer such
that
$$\displaystyle \max_k \{4\sigma_0^2(||a_{kk}||_\infty+2/\Delta t)/3\}<(N_0/L_{N_0})^2,$$
where $L_{N_0}=L(N_0)$ defined in (21). Then, for any $N\geq N_0$,  the discrete
operator defined by (24)-(25) is of positive type.
\end{lemma}
\begin{proof}~Firstly, for  $x_i\in(0,\tau)\cup(1-\tau,1)$, the
fourth order compact difference scheme is considered in this region.
The condition $\displaystyle \max_k
\{4\sigma_0^2(||a_{kk}||_\infty+2/\Delta t)/3\}<(N_0/L_{N_0})^2$ for any
$N\geq N_0$, where $L_{N_0}=L(N_0)$ as defined in (21), with the the
coefficients $q_i^{k,*},~r_i^{k,*},
 ~*=-,c,+$, ~~$k=1,2, \ldots, M$, defined by (\ref{4.6})-(\ref{4.7}) and
 the assumption $(4)$-(5), concludes the lemma.\\

Secondly, for $x_i\in[\tau,1-\tau]$ when  $\gamma
h_{\max}^2||\widehat{a}_{kk}||_\infty> \varepsilon$, the central
difference
scheme is considered. Hence the proof is trivial.\\

While in the opposite case, for $x_i \in [\tau,1-\tau]$  when
$\gamma h_{\max}^2||\widehat{a}_{kk}||_\infty \leqslant \varepsilon$, we
decompose $[\tau,1-\tau]:=[\tau, x_{N/2}]\cup[x_{N/2},1-\tau]$ and
study the sign of coefficients $q_i$'s in these two cases, separately.
First, when $x_i \in [\tau, x_{N/2}]$ the coefficients $q_i^{k,+}$
is clearly non-negative on generalized Shishkin mesh
$\overline{\Omega}_N^S$ while the coefficient $q_i^{k,-}$ will 
be non-negative when $2h_i-h_{i+1}\geqslant0$ for $N/4\leqslant i\leqslant N/2.$
The assertion is trivially true for $i=N/2$ because of uniform mesh (at symmetry). For $N/4\leqslant
i\leqslant N/2-1,$ we can write it\\
$$h_{i}\geqslant h_{i+1}-h_i,$$
\noindent follows if (cf. (23))
$$N\mathcal{K}'((i-1)/N)\geqslant\mathcal{K}^{''}((i+1)/N),$$
that is, if
$$\widetilde{w}(z)=3pz^2-6pz-12p+4\sigma N^2\geqslant0,$$
where $z=i-1-N/4\geqslant-1.$ It is not hard to verify that the discriminant of the quadratic function $\widetilde{w}$ is non-positive  if  $4\sigma N^2\geqslant 15p.$
Since $ \gamma h_{\max}^2||\widehat{a}_{kk}||_\infty\leqslant
\varepsilon$ definitely implies $4\sigma N^2\geqslant
 15p$, it follows that $\widetilde{w}(z)\geq0$ for all $z$. This completes the proof.\\

 Similar to this, we can prove
$ q_i^{k,+}\geqslant 0$ for $N/2\leqslant i\leqslant 3N/4.$ Thus
the condition $ \gamma h_{\max}^2||\widehat{a}_{kk}||_\infty\leqslant
\varepsilon$ with the  coefficients $q_i^{k,*},~r_i^{k,*},
 ~*=-,c,+$, ~~$k=1,2, \ldots, M$, defined by (\ref{4.6}),(31) and the assumption $(4)$-(5), concludes the lemma.

\end{proof}

\begin{remark}
It can be seen that the scheme $(24)$-$(25)$ on standard Shishkin
mesh does not satisfy the above Lemma 4.1 because the coefficients
$q_i^{k,*}$ are not always non-negative at the transition points,
due to the fact that the standard Shishkin mesh is very anisotropic
in nature. While if we consider the scheme $(24)$-$(25)$ on the
generalized Shishkin mesh $\overline{\Omega}_N^S$ then the
coefficients $q_i^{k,*}$ are always
non-negative. This is used in the proof of Lemma 4.1.\\
\end{remark}

Using the Lemma 4.1, the discretization operator defined  by
(24)-(25) is of positive type and it satisfies the following
discrete comparison principle.

\noindent \begin{lemma}(Discrete Comparison Principle) Let
$\widehat{\textbf{V}}$ and $\widehat{\textbf{W}}$ be two mesh
functions and satisfy
$[\widehat{\textbf{L}}^N_{x,\varepsilon}\widehat{\textbf{V}}]_i\geq
[\widehat{\textbf{L}}^N_{x,\varepsilon}\widehat{\textbf{W}}]_i$,~
$i=1,2,...,N-1,$~~$\widehat{\textbf{V}}_0\geq
\widehat{\textbf{W}}_0$ and $\widehat{\textbf{V}}_N\geq
\widehat{\textbf{W}}_N$, then $\widehat{\textbf{V}}_i\geq
\widehat{\textbf{W}}_i$~ $i=0,1,...,N.$\\
\end{lemma}
Using the above discrete comparison principle  we obtain the following discrete stability estimate.\\
\noindent
\begin{lemma}(Discrete Stability Estimate)
Let $\widehat{\textbf{V}}$  be the mesh function with
$\widehat{\textbf{V}}_0=\widehat{\textbf{V}}_N=0$. Then
$$\|\widehat{\textbf{V}}\|\leqslant C\|\widehat{\textbf{L}}^N_{x,\varepsilon}\widehat{\textbf{V}}\|$$
where $C$ is independent of $N$, $\Delta t$ and $\varepsilon$.
\end{lemma}

Let  $\Gamma_{\widehat{\textbf{\emph{u}}}^{n+1}}(x_i)$ be the
truncation error associated to the scheme (24)-(25) and is defined
by
$$\Gamma_{\widehat{\textbf{\emph{u}}}^{n+1}}(x_i)=[\widehat{\textbf{\emph{L}}}^N_{x,\varepsilon}(\widehat{\textbf{\emph{u}}}^{n+1}-\widehat{\textbf{\emph{U}}}^{n+1})]_i.$$

\begin{lemma}
Let $\widehat{\textbf{u}}^{n+1}$ be the solution of (\ref{3.3}) and
$\widehat{\textbf{U}}^{n+1}$ be the approximate solution of the
spatial discretized scheme (24)-(25). Let the hypothesis of Lemma
4.1 be satisfied. Then the global error satisfies
$$|\widehat{\textbf{u}}^{n+1}(x_i)-\widehat{\textbf{U}}^{n+1}_i|\leqslant \textbf{C} \Delta t(L/N)^4,$$
with the assumption that $L^{-4}\le C \Delta t$, where $C$ and $\textbf{C}$ are a positive constants independent of $N$, $\Delta t$ and $\varepsilon$.\end{lemma}

\begin{proof}\noindent If $\tau=1/4$, then the mesh $\overline{D}_N$ is
uniform, that is, $N^{-1}$ is very small respect to $\varepsilon$
and therefore a classical analysis can be used to prove the
convergence of the scheme. So, in the analysis we only consider the
case
$\tau=\sigma_0\sqrt{\varepsilon} L$.\\

\noindent The truncation error estimate
$\Gamma_{\widehat{\textbf{\emph{u}}}^{n+1}}(x_i)$ of the  scheme
$(24)$-$(25)$ on the generalized Shishkin mesh
$\overline{D}_N$ is discussed in the following cases.\\

$(I)$~When  $x_i\in (0, \tau)\cup(1-\tau, 1)$, we have
$h_i=h_{i+1}=4 \sigma_0 \sqrt{\varepsilon}N^{-1} L$. By Taylor expansions
we obtain
\begin{equation}
|\Gamma_{\widehat{\textbf{\emph{u}}}^{n+1}}(x_i)|\leqslant \textbf{\emph{C}} \varepsilon
\Delta t
h_i^4\left\|\frac{d^6\widehat{\textbf{\emph{u}}}^{n+1}}{dx^6}\right\|_{[x_{i-1},x_{i+1}]},
\end{equation}
\noindent Using $h_i=4 \sigma_0 \sqrt{\varepsilon}N^{-1} L$ and
$\left\|\frac{d^6\widehat{\textbf{\emph{u}}}^{n+1}}{dx^6}\right\|\leqslant
C\varepsilon^{-3}$, it follows that
\begin{equation}\label{4.9}
|\,[\Gamma(\textbf{\emph{u}})]_i\,|\leqslant \textbf{\emph{C}} \Delta t(L/N)^4 ,
\end{equation}

 $(II)$~When  $x_i\in [\tau, 1-\tau\ ],$ according
to the decomposition
$\widehat{\textbf{\emph{u}}}^{n+1}=\widehat{\textbf{\emph{v}}}^{n+1}+\widehat{\textbf{\emph{w}}}^{n+1}$,
split the truncation error into two parts to obtain
\begin{equation}\label{4.10}
|\,\Gamma_{\widehat{\textbf{\emph{u}}}^{n+1}}(x_i)\,|\le|\,\Gamma_{\widehat{\textbf{\emph{v}}}^{n+1}}(x_i)\,|+|\,\Gamma_{\widehat{\textbf{\emph{w}}}^{n+1}}(x_i)\,|.
\end{equation}

\noindent For the mesh points located in $[\tau,1-\tau]$, depending
on the relation between $h_{\max}$ and $\varepsilon$, the  scheme
$(24)$-$(25)$ is
 defined by the  combination modified high order non-equidistant  difference scheme and the central difference scheme. The error analysis
 for both cases are  given as follows.\\

$(i)$~For the case $\gamma h_{\max}^2||\widehat{a}_{kk}||_\infty
\leqslant \varepsilon$, suppose  $\textbf{\emph{g}}\in C^6[0,1]^M$, then
by Taylor expansions we obtain
\begin{equation}
|\Gamma_{\emph{g},k}(x_i)|\leqslant C \varepsilon \Delta
t(\emph{P}_{k,i}+\emph{Q}_{k,i}+\emph{R}_{k,i}),~~\mbox{for}
~~k=1,\ldots,M,
\end{equation}
\noindent where
 $$\emph{P}_{k,i}=(h_{i+1}-h_{i})^2||\emph{g}_{k}^{(4)}||_{[x_{i-1},x_{i+1}]},~~~\emph{Q}_{k,i}=|h_{i+1}-h_{i}|(h_{i+1}+h_{i})^2||\emph{g}_{k}^{(5)}||_{[x_{i-1},x_{i+1}]},$$
 $$\emph{R}_{k,i}=(h_i^4+h_{i+1}^4)||\emph{g}_k^{(6)}||_{[x_{i-1},x_{i+1}]}.$$

\noindent Using (\ref{4.a}) and (\ref{3.20}), we obtain the bound of
the truncation error with respect to the regular part
$\widehat{\textbf{\emph{v}}}^{n+1}$

\begin{equation}\label{4.11}
|\,\Gamma_{\widehat{\textbf{\emph{v}}}^{n+1}}(x_i)\,|\leqslant \textbf{\emph{C}} \Delta t
N^{-4}.
\end{equation}

\noindent Again using (\ref{4.a}) and (\ref{3.21}), we obtain the
bound of the truncation error with respect to the layer part
$\widehat{\textbf{\emph{w}}}^{n+1}$
\begin{equation}\label{4.12}
|\,\Gamma_{\widehat{\textbf{\emph{w}}}^{n+1}}(x_i)\,|\leqslant
\textbf{\emph{C}}\varepsilon^{-2} \Delta t
N^{-4}||B_{\varepsilon}||_{[x_{i-1},x_{i+1}].}
\end{equation}

\noindent For  $x_i\in[\tau,1-\tau]$,
$$||B_{\varepsilon}||_{[x_{i-1},x_{i+1}]}\leqslant
e^{(-x_{N/4-1}\sqrt{\beta/\varepsilon})}+e^{(-(1-x_{3N/4+1})\sqrt{\beta/\varepsilon})}=2e^{(-x_{N/4-1}\sqrt{\beta/\varepsilon})}$$
$$~~~~~~=2e^{(-\tau\sqrt{\beta/ \varepsilon})}e^{(h_{N/4}\sqrt{\beta/\varepsilon})}\leqslant Ce^{-4L}, ~~\mbox{where} ~\tau=\sigma_0\sqrt{\varepsilon} L, ~\sigma_0\geq 4/\sqrt{\beta}.$$

\noindent  Then $e^{-L}\leqslant L/N$  leads to
\begin{equation}\label{4.13}
||B_{\varepsilon}||_{[x_{i-1},x_{i+1}]}\leqslant C(L/N)^4.
\end{equation}
Using (\ref{4.13}) in (\ref{4.12}) with $\gamma
h_{\max}^2||\hat{a}_{kk}||_\infty \leqslant \varepsilon$, we get
\begin{equation}\label{4.14}
|\,\Gamma_{\widehat{\textbf{\emph{w}}}^{n+1}}(x_i)\,|\leqslant \textbf{\emph{C}}\Delta
t(L/N)^{4}.
\end{equation}

\noindent On combining (\ref{4.11}) and (\ref{4.14}) with
(\ref{4.10}), we obtain

\begin{equation}\label{19}
|\,\Gamma_{\widehat{\textbf{\emph{u}}}^{n+1}}(x_i)\,|\le \textbf{\emph{C}} \Delta
t(L/N)^4, ~~~\mbox{for}~~\gamma
h_{\max}^2||\widehat{a}_{kk}||_\infty \leqslant \varepsilon.\\
\end{equation}

$(ii)$~Now, for the case  $\gamma
h_{\max}^2||\widehat{a}_{kk}||_\infty > \varepsilon$, suppose
$\textbf{\emph{g}}\in C^4[0,1]^M$, then by Taylor expansions we
obtain

\begin{equation}
|\Gamma_{\emph{g},k}(x_i)|\leqslant C \varepsilon \Delta t(
\emph{Y}_{k,i}+\emph{Z}_{k,i}),~~~\mbox{for}~~k=1,\ldots,M,
\end{equation}
\noindent where

 $$ \emph{Y}_{k,i}=|h_{i+1}-h_i|||\emph{g}_k^{(3)}||_{[x_{i-1},x_{i+1}]},
 ~~~~\emph{Z}_{k,i}=h_{i+1}^2||\emph{g}_k^{(4)}||_{[x_{i-1},x_{i+1}]}.$$

\noindent Using  (\ref{4.a}) and (\ref{3.20}), we obtain the bound
of the truncation error with respect to the regular part
$\widehat{\textbf{\emph{v}}}^{n+1}$

\begin{equation}
|\,\Gamma_{\widehat{\textbf{\emph{v}}}^{n+1}}(x_i)\,|\leqslant \textbf{\emph{C}}
\varepsilon \Delta t N^{-2}.
\end{equation}

\noindent Now using the condition $\gamma
h_{\max}^2||\widehat{a}_{kk}||_\infty
> \varepsilon$,  we obtain
\begin{equation}\label{4.16}
|\,\Gamma_{\widehat{\textbf{\emph{v}}}^{n+1}}(x_i)\,|\leqslant \textbf{\emph{C}} N^{-4}.
\end{equation}

Note that in $(44)$  the  term $\Delta t$ disappears from the bound
for the error associated
with the regular part; this fact is important in order to impose the relation between the discretization parameters $\Delta t$ and $N$. \\

 To estimate the error with respect to the layer part
$\widehat{\textbf{\emph{w}}}^{n+1}$, suppose $\textbf{\emph{g}}\in
C^2[0,1]^M$, then using
$$ |\,\Gamma_\textbf{\emph{g}}(x_i)\,|\leqslant \textbf{\emph{C}} \varepsilon \Delta t  \|\textbf{\emph{g}}''\|_{[x_{i-1},x_{i+1}]},$$

\noindent we have
$$ |\,\Gamma_{\widehat{\textbf{\emph{w}}}^{n+1}}(x_i)\,|\leqslant \textbf{\emph{C}} \Delta t ||B_{\varepsilon}||_{[x_{i-1},x_{i+1}]}.$$

\noindent Using (\ref{4.13}), we have
\begin{equation}\label{4.17}
 |\,\Gamma_{\widehat{\textbf{\emph{w}}}^{n+1}}(x_i)\,|\leqslant \textbf{\emph{C}}\Delta t(L/N)^4.
\end{equation}

\noindent Combining (\ref{4.16}) and (\ref{4.17}) in (\ref{4.10}) with
the assumption such that $L^{-4}\le C \Delta t$, we obtain

\begin{equation}\label{4.18}
|\,\Gamma_{\widehat{\textbf{\emph{u}}}^{n+1}}(x_i)\,|\leqslant \textbf{\emph{C}} \Delta
t(L/N)^4, ~~~\mbox{for}~~\gamma
h_{\max}^2||\widehat{a}_{kk}||_\infty > \varepsilon.
\end{equation}

\noindent On combining the case $(I)$ and case $(II)$, we obtain the
truncation error estimate for the  scheme (24)-(25) on the
generalized Shishkin mesh $\overline{D}_N$ and it is given by
\begin{equation}\label{4.19}
|\,\Gamma_{\widehat{\textbf{\emph{u}}}^{n+1}}(x_i)\,|\leqslant \textbf{\emph{C}}\Delta
t(L/N)^4.
\end{equation}

Therefore, from the truncation error estimate (\ref{4.19}) and the
uniform stability result  given in Lemma 4.4, we conclude the lemma.
\end{proof}

\section{Total discretization }
In this section, we  write the total discretization  by combining
the time semidiscretization and spatial semidiscretization to
compute the approximate solution of $(1)$-$(3)$ and after that we
prove that the resultant scheme  is uniformly convergent of second
order in time and almost fourth order in spatial variable.
Concretely, the numerical approximate $\textbf{\emph{U}}^n_i$ of
$\textbf{\emph{u}}(x_i,n\Delta t)$ for $i=1,\ldots,N$ and
$n=0,1,\ldots,T/\Delta t$, are obtained by the following totally
discrete scheme
\begin{equation}\label{5.1}
\left\{%
\begin{array}{ll}
\textbf{\emph{U}}^0_i=\textbf{0},
~~\textbf{\emph{L}}^N_{x,\varepsilon}\textbf{U}^0_{i}=\textbf{\emph{L}}_{x,\varepsilon}\textbf{\emph{u}}(x_i,0),
 ~~i=0(1)N,\\
 \mbox{[}\widehat{\textbf{\emph{L}}}^N_{x,\varepsilon}\textbf{\emph{U}}^{n+1}\mbox{]}_i=\mbox{[}\emph{\textbf{Q}}\textbf{\emph{F}}^n\mbox{]}_i,
~~\mbox{for}~~i=1(1)N-1,\\
\textbf{\emph{U}}^{n+1}_0=\textbf{0},~~\textbf{\emph{U}}^{n+1}_N=\textbf{0},~~ \mbox{for}~~n=0,1,\ldots,T/\Delta t-1,\\
\end{array}
\right.
\end{equation}
where
\begin{equation}\label{discoper}
[\widehat{\textbf{\emph{L}}}^N_{x,\varepsilon}\textbf{\emph{U}}^{n+1}]_i:=\left(%
\begin{array}{c}
  R(U^{n+1}_1)~+~\frac{\Delta t}{2}Q(a_{12}U^{n+1}_2)+...+\frac{\Delta t}{2}Q(a_{1M}U^{n+1}_M)\\

R(U^{n+1}_2)~+~\frac{\Delta t}{2}Q(a_{21}U^{n+1}_1)+...+\frac{\Delta t}{2}Q(a_{2M}U^{n+1}_M)\\
\vdots
\\
R(U^{n+1}_M)~+~\frac{\Delta t}{2}Q(a_{M1}U^{n+1}_1)+...+\frac{\Delta t}{2}Q(a_{MM-1}U^{n+1}_{M-1})\\
\end{array}%
\right)_i,
\end{equation}

\begin{equation}
[\emph{\textbf{Q}}\textbf{\emph{F}}^n]_i:=\left(%
\begin{array}{c}
  Q(F_1^n) \\

  Q(F_2^n)  \\
  \vdots
  \\
  Q(F_M^n) \\

\end{array}%
\right)_i,
\end{equation}
$$\textbf{\emph{F}}^{n}(x_i)=\textbf{\emph{U}}_i^n+\frac{\Delta
t}{2}(-\textbf{\emph{L}}^N_{x,\varepsilon}\textbf{\emph{U}}^{n}_{i}+\textbf{\emph{f}}(x_i,t_n)+\textbf{\emph{f}}(x_i,t_{n+1})),$$
for $i=1(1)N-1,
~n=0,1,\ldots,T/\Delta t-1,$
$$\hspace*{-0.75cm}\textbf{\emph{L}}^N_{x,\varepsilon}\textbf{\emph{U}}^{n+1}_{i}=-\textbf{\emph{L}}^N_{x,\varepsilon}\textbf{\emph{U}}^{n}_{i}-2\frac{\textbf{\emph{U}}^{n+1}_{i}-\textbf{\emph{U}}^{n}_{i}}{\Delta
t}+
\textbf{\emph{f}}(x_i,t_n)+\textbf{\emph{f}}(x_i,t_{n+1}),$$ \noindent and
\begin{equation}
\left\{%
\begin{array}{lcl}
R(V_{k,i}) &=& r_i^{k,-}V_{k,i-1}+r_i^{k,c}V_{k,i}+r_i^{k,+}V_{k,i+1},\\
Q(V_{k,i}) &=& q_i^{k,-}V_{k,i-1}+q_i^{k,c}V_{k,i}+q_i^{k,+}V_{k,i+1}.
\end{array}\right.
\end{equation}

\noindent The coefficients $q_i^{k,*},~
 ~*=-,c,+$ and $r_i^{k,*},~
 ~*=-,c,+$, ~~$k=1,2, \ldots, M$ are defined as in Section 4.\\

\begin{theorem}
Let $\textbf{u}$ be the exact solution of (1)-(3) and let
$\{\textbf{U}^{n+1}_i\}$ be the numerical solution of the scheme
(48). Under the hypothesis of Lemma 4.1, the global error~
$\textbf{u}(x,t_{n+1})-\textbf{U}^{n+1}~$ at the time $t_{n+1}$
satisfies
$$||\textbf{u}(x_i,t_{n+1})-\textbf{U}_i^{n+1}||_{\overline{D}_N}\leqslant C((\Delta t)^2+ (L/N)^4),$$
 with the assumption that $L^{-4}\le C \Delta t$, where $C$ is a positive constant independent of $N$, $\Delta t$ and $\varepsilon$.
\end{theorem}
\begin{proof}
  The global error
$\textbf{\emph{u}}(x_i,t_{n+1})-\textbf{\emph{U}}^{n+1}_i~$ of the
totally discrete scheme at the time $t_{n+1}$ can be split in the
form
\begin{equation}
\textbf{\emph{u}}(x_i,t_{n+1})-\textbf{\emph{U}}_i^{n+1}~\leqslant
(\textbf{\emph{u}}(x_i,t_{n+1})-\widehat{\textbf{\emph{u}}}^{n+1}(x_i))~+(\widehat{\textbf{\emph{u}}}^{n+1}(x_i)-\widehat{\textbf{\emph{U}}}_i^{n+1}~
)\end{equation}
$~~~~~~~~~~\hspace*{6.0cm}~~~~~~~~~~~~~~+(\widehat{\textbf{\emph{U}}}_i^{n+1}-\textbf{\emph{U}}_i^{n+1}).$\\
On combining the  result from the Lemma 3.1 and  Lemma 4.5 with
(52), we obtain
\begin{equation}
||\textbf{\emph{u}}(x_i,t_{n+1})-\textbf{\emph{U}}_i^{n+1}~||_{\overline{D}_N}\leqslant
C ((\Delta t)^3+ \Delta
t(L/N)^4)+||\widehat{\textbf{\emph{U}}}_i^{n+1}-\textbf{\emph{U}}_i^{n+1}||_{\overline{D}_N}.
\end{equation}
To bound the term
$||\widehat{\textbf{\emph{U}}}_i^{n+1}-\textbf{\emph{U}}_i^{n+1}||_{\overline{D}_N}$
, we consider that
$\widehat{\textbf{\emph{U}}}^{n+1}-\textbf{\emph{U}}^{n+1}$ can be
written as the solution of one step of $(48)$ with starting value
$\textbf{\emph{u}}(x_i,t_{n})-\textbf{\emph{U}}_i^{n}$, taking the
source term $\textbf{\emph{f}}$ equal to zero together  with zero
boundary conditions. Then it follows that
\begin{equation}
\widehat{\textbf{\emph{U}}}_i^{n+1}-\textbf{\emph{U}}_i^{n+1}=\emph{\textbf{R}}_N(\textbf{\emph{u}}(x_i,t_{n})-\textbf{\emph{U}}_i^{n}),
\end{equation}
where $R_N$ is a linear operator, called the transition operator
associated to the totally discrete scheme (48). Using this with (53)
we obtain a recursive argument as
\begin{equation}
||\textbf{\emph{u}}(x_i,t_{n+1})-\textbf{\emph{U}}_i^{n+1}~||_{\overline{D}_N}\leqslant
C\sum_{k=1}^{n}||\emph{\textbf{R}}_N^{n-k}||_{\overline{D}_N}((\Delta
t)^3+ \Delta t(L/N)^4).
\end{equation}
To get the required result for the uniform convergence of
totally discrete scheme a sufficient condition is that
$$||\emph{\textbf{R}}^{j}_N||_{\overline{D}_N}\leqslant C,~~~j=1,\ldots,n.$$ By assuming
the uniform boundedness condition on power of discrete transition
operator $\emph{\textbf{R}}_N$ with (53) (see the Remark 5.2 below) and  the hypothesis of Lemma 4.5 that $L^{-4}\leqslant C\Delta
t$, we conclude the main results of this section.
\end{proof}
\begin{remark}
We assume here the uniform boundedness condition as a conjuncture
holds for the transition operator $\textbf{R}_N$, as the
theoretical proof of this is an open problem so far in the
literature. Some partial results in this direction can be obtained
by using a result by Palencia \cite{pal}, but for the present
problem this would require an $\varepsilon$-uniform estimate of the
resolvent of the spatial operator
$\textbf{L}_{x,\varepsilon}$. Here due to lack of available
theoretical result in this direction we assume the uniform
boundedness of the power of discrete transition operator
$\textbf{R}_N$ as a conjuncture. For the support of this
conjuncture  we show some numerical evidence for the spectral radius
of $\textbf{R}_N$. From the numerical results of the Tables 2, 4, and 6 we
observe that the spectral radius of $\textbf{R}_N$ is
strictly less than one and it stabilize as the singular perturbation
parameter $\varepsilon$ becomes small.
\end{remark}

\begin{remark}
Theorem 5.1 proves almost fourth order uniform convergence of the
method in spatial variable under the relation  $L^{-4}\leqslant C \Delta t$. Nevertheless,
from the numerical point of view in Section 6, this condition is an
artificial relation that we never needed in the experiments. Note
that this relation appeared when we prove the convergence of the
regular components in regular region, see eq. (44) in Section 4.
\end{remark}
\section{Numerical experiments}
The proposed  method is implemented on three test examples.  In all
the cases we begin with total number of nodal points $N=64$ and the
time step $\Delta t=0.5$. The maximum error at the nodal points is
calculated for the different values of
$\epsilon$ and N.\\

\noindent \textbf{Example~1:}~Consider the following  system of two
coupled singularly perturbed parabolic problem
\begin{eqnarray*}
\frac{\partial u_1}{\partial t}- \varepsilon \frac{\partial^2 u_1}{\partial x^2}+(2+x)u_1-(1+x)u_2 &=& x^2(1-x)^2,\\
\frac{\partial u_2}{\partial t}- \varepsilon \frac{\partial^2 u_2}{\partial x^2}+(e^x+1)u_2-(1+x)u_1 &=& x^2(1-x)^2,
\end{eqnarray*}
for ~$(x,t)\in (0,1)\times(0,1],$ with the initial-boundary
conditions
$$\textbf{\emph{u}}(x,0)=\textbf{0},~~~~~x\in(0,1),$$
$$\textbf{\emph{u}}(0,t)=\textbf{\emph{u}}(1,t)=\textbf{0},~~~~t\in[0,1],$$
 \noindent and the exact solution is not known.\\

\noindent \textbf{Example~2:}~Consider the following  system of
three coupled singularly perturbed parabolic problem
\begin{eqnarray*}
\frac{\partial u_1}{\partial t}- \varepsilon \frac{\partial^2 u_1}{\partial x^2}+3u_1-(1-x)u_2-(1-x)u_3 &=& 16x^2(1-x)^2,\\
\frac{\partial u_2}{\partial t}- \varepsilon \frac{\partial^2 u_2}{\partial x^2}+(4+x)u_2-2u_1-u_3 &=& t^3,\\
\frac{\partial u_3}{\partial t}- \varepsilon \frac{\partial^2 u_3}{\partial x^2}+(6+x)u_3-2u_1-3u_2 &=& 16x^2(1-x)^2,
\end{eqnarray*}
 for $(x,t)\in D$, with the initial-boundary conditions
$$\textbf{\emph{u}}(x,0)=\textbf{0},~~~~~x\in(0,1),$$
$$\textbf{\emph{u}}(0,t)=\textbf{\emph{u}}(1,t)=\textbf{0},~~~~t\in[0,1],$$
 \noindent and the exact solution is not known.\\

\noindent \textbf{Example~3:}~Consider the following  system of two
coupled singularly perturbed parabolic problem
\begin{eqnarray*}
\frac{\partial u_1}{\partial t}- \varepsilon
\frac{\partial^2 u_1}{\partial x^2}+2u_1-u_2 &=& 1,\\
\frac{\partial u_2}{\partial t}- \varepsilon \frac{\partial^2 u_2}{\partial x^2}+2u_2-u_1 &=& 1,
\end{eqnarray*}
for $(x,t)\in (0,1)\times(0,1],$ with the initial-boundary
conditions
$$\textbf{\emph{u}}(x,0)=\textbf{0},~~~~~x\in(0,1),$$
$$\textbf{\emph{u}}(0,t)=\textbf{\emph{u}}(1,t)=\textbf{0},~~~~t\in[0,1],$$
\noindent   and the exact solution is not known.\\

\begin{table}[h] 
\caption{~Maximum error  and numerical rate of convergence of the
present method  with uniform step size $\Delta t$ in time direction
and generalized Shishkin mesh $S(L)$ with $L=L^*$ in
spatial direction for the Example 1.}
{\mbox{\tabcolsep=2pt\begin{tabular}{@{}ccccccc@{}}  \hline\\
$\varepsilon=2^{-k}$ & $N =64$ &$ N =128$ &$ N =256$ &$N =512$&
$N=1024$ \\
$$ & $\Delta t=0.5$ &$ \Delta t=0.5/4$ &$ \Delta t=0.5/4^2$ &$\Delta t=0.5/4^3$& $\Delta t=0.5/4^4$
\\[.8ex]
\hline
\\[3pt]
 $k$=4  &8.82E-04&     4.42E-05 &   2.71E-06  &  1.69E-07 &   1.06E-08\\[3pt]
      & 4.32 &  4.03&   4.00&   4.00  \\[2pt]
8  &4.19E-04  &  2.57E-05  &  1.60E-06  &  1.00E-07  &  6.26E-09\\[3pt]
      & 4.03 &  4.00&   4.00&   4.00  \\[2pt]
12  &3.91E-04 &   2.43E-05  &  1.52E-06  &  9.52E-08  &  5.95E-09\\[3pt]
      & 4.01 &  3.99&   4.00&   4.00  \\[2pt]
16  &3.88E-04 &   2.41E-05 &   1.52E-06  &  9.51E-08 & 6.00E-09\\[3pt]
      & 4.01 &  3.99&   4.00&   3.99  \\[2pt]
20  &3.86E-04 &   2.41E-05  &  1.52E-06   & 9.48E-08 &5.93E-09\\[3pt]
      & 4.00 &  3.99&   4.00&   4.00  \\[2pt]
24  &3.86E-04 &   2.41E-05  &  1.52E-06   & 9.48E-08 &5.93E-09\\[3pt]
      & 4.00 &  3.99&   4.00&   4.00  \\[2pt]
28  &3.86E-04 &   2.41E-05  &  1.52E-06   & 9.48E-08 &5.93E-09\\[3pt]
      & 4.00 &  3.99&   4.00&   4.00  \\[2pt]
32  &3.86E-04 &   2.41E-05  &  1.52E-06   & 9.48E-08 &5.93E-09\\[3pt]
      & 4.00 &  3.99&   4.00&   4.00  \\[2pt]
\hline
\\[2pt]
    $\emph{E}_{ N, \bigtriangleup t}$& 8.82E-04&     4.42E-05 &   2.71E-06  &  1.69E-07 &   1.06E-08\\[3pt]
    $p^N$&  4.32 &  4.03&   4.00&   4.00 \\[2pt]
\hline
\end{tabular}
\label{Neigs2}
 }}{}
\end{table}
\begin{table}[] 
\caption{~Spectral radius of the transition operator
$\emph{\textbf{R}}_N$ for the Example 1.}
{\mbox{\tabcolsep=2pt\begin{tabular}{@{}ccccccc@{}} \hline\\
$\varepsilon=2^{-k}$ & $N =64$ &$ N =128$ &$ N =256$ &$N =512$&
$N=1024$ \\
$$ & $\Delta t=0.5$ &$ \Delta t=0.5/4$ &$ \Delta t=0.5/4^2$ &$\Delta t=0.5/4^3$& $\Delta t=0.5/4^4$
\\[.8ex]
\hline
 $k$=4  &0.40350&     0.80792&   0.94825  &  0.98681 &   0.99670\\[3pt]
8  &0.57817  &  0.87118  &  0.96616 &  0.99143 &  0.99785\\[3pt]
12  &0.59223 &   0.87965  &  0.96849  &  0.99203  &  0.99800\\[3pt]
16  &0.59810 &   0.88169 &   0.96905  &  0.99217 & 0.99804\\[3pt]
20  &0.59953 &   0.88219  &  0.96919   & 0.99221 &0.99805\\[3pt]
24  &0.59989 &   0.88234  &  0.96923   & 0.99222 &0.99805\\[3pt]
28  &0.59997 &   0.88235  &  0.96923   & 0.99222 &0.99805\\[3pt]
32  &0.59999 &   0.88235  &  0.96923   & 0.99222 &0.99805\\[3pt]
\hline
\end{tabular}
\label{Neigs2}
 }}{}
\end{table}

\begin{table}[h] 
\caption{~Maximum error  and numerical rate of convergence of the
present method  with uniform step size $\Delta t$ in time direction
and generalized Shishkin mesh $S(L)$ with $L=L^*$ in
spatial direction for the Example 2.}
{\mbox{\tabcolsep=2pt\begin{tabular}{@{}ccccccc@{}} \hline\\
$\varepsilon=2^{-k}$ & $N =64$ &$ N =128$ &$ N =256$ &$N =512$&
$N=1024$ \\
$$ & $\Delta t=0.5$ &$ \Delta t=0.5/4$ &$ \Delta t=0.5/4^2$ &$\Delta t=0.5/4^3$& $\Delta t=0.5/4^4$
\\[.8ex]
\hline
\\[3pt]
 $k$=4  &4.48E-02 &   4.09E-03  &  2.17E-04   &  1.35E-05  &  8.44E-07
\\[3pt]
      & 3.45 &   4.23 &   4.01 &  4.00
\\[2pt]
8  &4.44E-02 &   3.58E-03  &  1.97E-04   &  1.22E-05   & 7.63E-07
\\[3pt]
      & 3.63 &  4.18&   4.01&   4.00  \\[2pt]
12  &4.43E-02  &  3.54E-03  &  1.95E-04   & 1.21E-05   & 7.58E-07
\\[3pt]
      & 3.65 &  4.18&   4.01&   4.00  \\[2pt]
16  &4.43E-02  &  3.54E-03  &  1.95E-04   & 1.21E-05   & 7.58E-07
\\[3pt]
      & 3.65 &  4.18&   4.01&   4.00  \\[2pt]
      20  &4.43E-02  &  3.54E-03  &  1.95E-04   & 1.21E-05   & 7.58E-07
\\[3pt]
      & 3.65 &  4.18&   4.01&   4.00  \\[2pt]
      24  &4.43E-02  &  3.54E-03  &  1.95E-04   & 1.21E-05   & 7.58E-07
\\[3pt]
      & 3.65 &  4.18&   4.01&   4.00  \\[2pt]
      28  &4.43E-02  &  3.54E-03  &  1.95E-04   & 1.21E-05   & 7.58E-07
\\[3pt]
      & 3.65 &  4.18&   4.01&   4.00  \\[2pt]
32  &4.43E-02  &  3.54E-03  &  1.95E-04   & 1.21E-05   & 7.58E-07
\\[3pt]
      & 3.65 &  4.18&   4.01&   4.00  \\[2pt]

\hline
\\[2pt]
    $\emph{E}_{ N, \bigtriangleup t}$&4.48E-02 &   4.09E-03  &  2.17E-04   &  1.35E-05  &  8.44E-07\\[3pt]
    $p^N$&   3.45 &   4.23 &   4.01 &  4.00\\[2pt]
\hline
\end{tabular}
\label{Neigs2}
 }}{}
\end{table}

\begin{table}[] 
\caption{~Spectral radius of the transition operator
$\emph{\textbf{R}}_N$ for the Example 2.}
{\mbox{\tabcolsep=2pt\begin{tabular}{@{}ccccccc@{}} \hline\\
$\varepsilon=2^{-k}$ & $N =64$ &$ N =128$ &$ N =256$ &$N =512$&
$N=1024$ \\
$$ & $\Delta t=0.5$ &$ \Delta t=0.5/4$ &$ \Delta t=0.5/4^2$ &$\Delta t=0.5/4^3$& $\Delta t=0.5/4^4$
\\[.8ex]
\hline
 $k$=4  &0.01380&     0.68153&   0.90960  &  0.97661 &   0.99410\\[3pt]
8  &0.25983  &  0.72130  &  0.93118 &  0.98234 &  0.99556\\[3pt]
12  &0.32590 &   0.77295  &  0.93796  &  0.98412  &  0.99601\\[3pt]
16  &0.34288 &   0.78200 &   0.94065  &  0.98482 & 0.99618\\[3pt]
20  &0.35146 &   0.78561  &  0.94172   & 0.98510 &0.99625\\[3pt]
24  &0.35448&   0.78705  &  0.94214  & 0.98526 &0.99628\\[3pt]
28  &0.35577 &   0.78762  &  0.94231   & 0.98528 &0.99628\\[3pt]
32  &0.35625 &   0.78784  &  0.94231   & 0.98528 &0.99628\\[3pt]
\hline
\end{tabular}
\label{Neigs2}
 }}{}
\end{table}

\begin{table}[h] 
\caption{~Maximum error  and numerical rate of convergence of the
present method  with uniform step size $\Delta t$ in time direction
and generalized Shishkin mesh $S(L)$ with $L=L^*$ in
spatial direction for the Example 3.}
{\mbox{\tabcolsep=2pt\begin{tabular}{@{}ccccccc@{}} \hline\\
$\varepsilon=2^{-k}$ & $N =64$ &$ N =128$ &$ N =256$ &$N =512$&
$N=1024$\\
$$ & $\Delta t=0.5$ &$ \Delta t=0.5/4$ &$ \Delta t=0.5/4^2$ &$\Delta t=0.5/4^3$& $\Delta t=0.5/4^4$
\\[.8ex]
\hline
\\[3pt]
 $k$=4  &3.32E-02 &   7.75E-03  &  1.91E-03 &   4.76E-04  &  9.52E-05\\[3pt]
      & 2.10 &   2.02 &    2.00 &    2.32
 \\[2pt]
8  & 3.29E-02 &    7.75E-03 &   1.91E-03  &  4.76E-04 &   9.53E-05 \\[3pt]
      & 2.08 &    2.02&   2.00& 2.32
  \\[2pt]
12  & 1.67E-02 &   2.41E-03  &  5.17E-04 &   1.24E-04  &  2.49E-05\\[3pt]
      & 2.80&   2.22&   2.06&  2.32 \\[2pt]
16  &1.67E-02  &  2.16E-03  &  4.41E-04   & 1.24E-04  &  2.47E-05\\[3pt]
      & 2.95 & 2.29 & 1.84 &  2.32 \\[2pt]
20  &1.67E-02  &  2.16E-03  &  4.41E-04   & 1.24E-04  &  2.47E-05\\[3pt]
      & 2.95 & 2.29 & 1.84 &  2.32 \\[2pt]
24  &1.67E-02  &  2.16E-03  &  4.41E-04   & 1.24E-04  &  2.47E-05\\[3pt]
      & 2.95 & 2.29 & 1.84 &  2.32 \\[2pt]
28  &1.67E-02  &  2.16E-03  &  4.41E-04   & 1.24E-04  &  2.47E-05\\[3pt]
      & 2.95 & 2.29 & 1.84 &  2.32 \\[2pt]
32  &1.67E-02  &  2.16E-03  &  4.41E-04   & 1.24E-04  &  2.47E-05\\[3pt]
      & 2.95 & 2.29 & 1.84 &  2.32 \\[2pt]
\hline
\\[2pt]
    $\emph{E}_{ N, \bigtriangleup t}$ &3.32E-02 &   7.75E-03  &  1.91E-03 &   4.76E-04  &  9.52E-05\\[3pt]
    $p^N$&  2.10 &   2.02 &    2.00 &    2.32 \\[2pt]
\hline
\end{tabular}
\label{Neigs2}
 }}{}
\end{table}

\begin{table}[] 
\caption{~Spectral radius of the transition operator
$\emph{\textbf{R}}_N$ for the Example 3.}
{\mbox{\tabcolsep=2pt\begin{tabular}{@{}ccccccc@{}} \hline\\
$\varepsilon=2^{-k}$ & $N =64$ &$ N =128$ &$ N =256$ &$N =512$&
$N=1024$ \\
$$ & $\Delta t=0.5$ &$ \Delta t=0.5/4$ &$ \Delta t=0.5/4^2$ &$\Delta t=0.5/4^3$& $\Delta t=0.5/4^4$
\\[.8ex]
\hline
 $k$=4  &0.42429&     0.88209&   0.95077  &  0.98745 &   0.99685\\[3pt]
8  &0.58776  &  0.88234  &  0.96806 &  0.99192 &  0.99792\\[3pt]
12  &0.59923 &   0.88235  &  0.96916  &  0.99220  &  0.99793\\[3pt]
16  &0.59995 &   0.88235 &   0.96923  &  0.99222 & 0.99794\\[3pt]
20  &0.60000 &   0.88235 &   0.96923  &  0.99222 & 0.99794\\[3pt]
24  &0.60000 &   0.88235 &   0.96923  &  0.99222 & 0.99794\\[3pt]
28  &0.60000 &   0.88235 &   0.96923  &  0.99222 & 0.99794\\[3pt]
32  &0.60000 &   0.88235 &   0.96923  &  0.99222 & 0.99794\\[3pt]
\hline
\end{tabular}
\label{Neigs2}
 }}{}
\end{table}
As the exact solution is not known for these examples, we estimate the maximum nodal
error, $\widetilde{\emph{E}}_{\varepsilon, N, \bigtriangleup
t}=\displaystyle \max_{\forall
i,n}\widetilde{\textbf{\emph{e}}}_{\varepsilon}^{N, \bigtriangleup
t}(i,n\Delta t),$ for different values of $\varepsilon $ and $N$
where $\widetilde{\textbf{\emph{e}}}_{\varepsilon}^{ N,
\bigtriangleup t}(i,n\Delta
t)=|\textbf{\emph{u}}_N(x_i,t_n)-\widetilde{\textbf{\emph{u}}}_{N}(x_i,t_n)|$.
We use a variant of the double mesh principle, assume
$\textbf{\emph{u}}_{N}(x_i,t_n)$ denotes the numerical solution at
the nodal point $(x_i,t_n)$ on the tensor product mesh of the
generalized Shishkin mesh $\overline{\Omega}_N^S$ with $N+1$ nodal points in spatial
direction and  a uniform mesh of step size $\Delta t$  in time
direction, and $\widetilde{\textbf{\emph{u}}}_{N}(x_i,t_n)$ denotes
the numerical solution at the nodal point $(x_i,t_n)$ on the tensor
product mesh
 $\{(\widehat{x}_i,\widehat{t}_n)\}$ that  contains the
mesh points of the original mesh and their
midpoints.\\

 In the
standard way, we estimate the
  classical  convergence rate, for each fixed
 $\varepsilon$, by

\[p^N_{\varepsilon}=\frac
 {\ln(\widetilde{\emph{E}}_{\varepsilon, N, \bigtriangleup
t})-\ln(\widetilde{\emph{E}}_{\varepsilon, 2N, \bigtriangleup
t/4})}{\ln2},
\]
 and the  parameter-robust   convergence rate $p^N$  by
\[p^N=\frac
 {\ln(\emph{E}_{ N, \bigtriangleup
t})-\ln(\emph{E}_{ 2N, \bigtriangleup t/4})}{\ln2},
\]
where $\emph{E}_{ N, \bigtriangleup t}=\displaystyle
\max_{\forall \varepsilon}\widetilde{\emph{E}}_{\varepsilon, N,
\bigtriangleup t}$.\\

Using $L<\ln N$ instead of $\ln N$; this means we are trying to
bring the point $x_1$ closer to $x=0$ and this provides the higher
density of the mesh points in the layers. The motivation for this is
the fact that the better performance of the mesh $S(L)$ can be governed by
the high density of mesh points in the layers. The smallest value of
$L$ is chosen to be  $L^*=L^*(N)$ which satisfies
$$e^{-L^*}=L^*/N.$$

For the different values of  $N$ and $\varepsilon$, Table $1$, Table
$2$, and Table $3$ represent the maximum error
$\widetilde{\emph{E}}_{\varepsilon, N, \bigtriangleup t}$ and the
classical rate of convergence $p^N_{\varepsilon}$ of the present
method for the
 Example 1, Example 2, and Example 3, respectively. The last two rows in each of the tables
(Table 1, Table 3, and Table 5) represent the maximum error with
respect to each nodal point for all value of $\varepsilon$, that is
$\emph{E}_{ N, \bigtriangleup t}$ ; and  the parameter-robust
numerical rate of convergence $p^N$.
\\

To show the numerical evidence for the  uniform stability of the
transition operator $\emph{\textbf{R}}_N$, we calculate the
spectral radius of $\emph{\textbf{R}}_N$ for different value of $N$, $\Delta t$ and $\varepsilon$. Table 2, Table 4, and Table 6 display
the spectral radius of this operator for the Example 1, Example 2,
and Example 3, respectively. We clearly observe that the spectral
radius for all value of $N$, $\Delta t$ and $\varepsilon$ is always
strictly less than one. Moreover, we observe that the spectral
radius stabilized for the small value of singular perturbation
parameter $\varepsilon$. This stabilization of spectral radius for
small value of $\varepsilon$ indicates the uniform
stability of the operator $\emph{\textbf{R}}_N$.\\

Observe that the  data in Example 3 does not satisfy the zeroth
order compatibility conditions at the nodal points $(0,0)$ and
$(1,0)$. Moreover,  Table 5 shows the low order of accuracy of the
present method for the Example $3$ in comparison with the numerical
results presented in Table 1 and Table 3 for the Example $1$ and Example $2$, respectively; in which the sufficient compatibility conditions are
satisfied. From this one can infer that, in practice some of the theoretical compatibility conditions seems to be very necessary for high order convergence of the present method. 
Clearly the numerical results presented in Table 1 and Table 3
 verify our theoretical
 results.\\

 Previously, the Crank-Nicolson
method has been used
 in the framework of scalar singulary perturbed problem, for
instance, in \cite{clav05} to solve one dimensional parabolic
problems of convection diffusion type. Recently, Clavero et al. \cite{clav06}
considered the Crank-Nicolson method on uniform mesh in time
discretization and the central difference scheme on standard
Shishkin mesh in spatial discretization for a system of two coupled time
dependent singularly perturbed reaction-diffusion problems. In this article, to obtain a high order robust approximation we considered the Crank-Nicolson method in time direction and a hybrid scheme which is a
suitable combination of fourth order compact difference scheme (or
HODIE scheme ) and standard central difference scheme on a
generalized Shishkin mesh in spatial direction. Here it is interesting  to see how the HODIE
technique permits to obtain a uniformly convergent method having
order bigger than two in spatial direction. Earlier, the HODIE scheme for scalar
singularly perturbed reaction-diffusion problems has been considered in Clavero and Gracia \cite{clav07},  and it is proved that
the scheme is third order uniformly convergent on standard
Shishkin mesh.  But the extension of new HODIE scheme
on standard Shishkin mesh is not possible in the case of system
of coupled reaction-diffusion problems. It can be seen that the
coefficients $q_i^{k}$'s in (24)-(25) is not always positive at the
transition points, due to the fact that standard Shishkin mesh
is very anisotropic in nature. This shows that the operator in
(24)-(25) is not of positive type on standard Shishkin mesh. At
the moment, when $N^{-1}<\sqrt{\varepsilon}$ we can not find a
difference scheme of positive type which is high order uniformly
convergent on standard Shishkin mesh for system of coupled
reaction-diffusion problems.  To avoid this, one can use the central
difference scheme in the regular region $[\tau,1-\tau]$ and the
fourth order compact difference scheme in $(0,\tau)\cup(1-\tau,1)$.
But  this combination gives only second order uniformly convergent
result.  In order to increase the order of convergence 
and to maintain the positivity of the present discrete operator in
(24)-(25), we consider a generalized
Shishkin mesh  instead of standard Shishkin mesh. The Lemma 4.1 shows that the
discrete operator in $(24)$-(25) on a generalized Shishkin mesh is
of positive type and the analysis in Section 4 shows that the scheme
$(24)$-(25) is  almost fourth order uniformly convergent with
respect to the perturbation parameter $\varepsilon$.  Here we also want to point out one more  benefit of generalized
Shishkin mesh over standard Shishkin mesh in the
numerical methods presented in \cite{clav06} and \cite{gra07 } for
parabolic reaction diffusion systems. It is proved that the numerical methods presented
in  \cite{clav06} and \cite{gra07 } have almost second order uniform
convergence under the theoretical relation $N^{-q}\leqslant C\Delta t$,
where $0<q<1$. Note that the theoretical relation appeared in the analysis when
the barrier function technique was used to prove the second order
convergence of the regular component on standard Shishkin mesh.
While if we use generalized Shishkin mesh instead of standard
Shishkin mesh in \cite{clav06} and \cite{gra07 } then we can claim
almost second order uniform convergence in spatial variable without
any theoretical relation by using the same analysis technique.

\section{Conclusions}
We presented  a high order parameter-robust numerical method
for a  system of $(M\ge2)$ coupled  singularly perturbed parabolic
 reaction-diffusion  problem (1)-(3). The problem is
discretized using the Crank-Nicolson method on an uniform mesh in time
direction and a suitable combination of the fourth order compact
difference scheme and the central difference scheme on a
generalized Shishkin mesh in spatial direction.
The essential idea in this method is to use a generalized Shishkin
mesh  in order to attain a high order
parameter-robust convergence in spatial variable.  The fine parts of standard Shishkin mesh  and generalized Shishkin mesh
 are identical, but the coarse part of
generalized Shishkin mesh is a smooth continuation of the fine mesh and is
no longer equidistant. Using this fact we proved that the present
method is second order uniformly convergent in time  and  almost
fourth order uniformly convergent in spatial variable, if the
discretization parameters satisfy a non-restrictive relation.
Numerical experiments are presented to validate
 the theoretical results and also the results of the experiments indicate that the relation between the discretization parameters is
 not necessary in practice.

\end{document}